\documentclass{article}

\usepackage{amssymb}
\usepackage{amsmath}
\usepackage{amsthm}

\newtheorem{theorem}{Theorem}
\newtheorem{lemma}{Lemma}[section]

\newtheorem{claim}{Claim}[section]
\theoremstyle{definition}

\newtheorem{defin}{Definition}[section]
\newtheorem{remark}{Remark}[section]

\newcommand{\la}{\lambda}

\newcommand{\fii}{\varphi}

\newcommand{\R}{\mathbb{R}}

\newcommand{\cS}{\mathcal{S}}
\newcommand{\cL}{\mathcal{L}}
\newcommand{\cD}{\mathcal{D}}

\begin{document}

\title{Nodal geometry of graphs on surfaces}

\author{Yong Lin, G\'abor Lippner, Dan Mangoubi, Shing-Tung Yau}
\date{April 14, 2010}

\maketitle

\begin{abstract}   
We prove two mixed versions of the Discrete Nodal Theorem of Davies
et.\ al.\ ~\cite{Davies} for bounded degree graphs, and for three-connected graphs of fixed genus
$g$. Using this we can show that for a three-connected graph
satisfying a certain volume-growth condition, the multiplicity of
the $n$th Laplacian eigenvalue is at most $2\left[
  6(n-1) + 15(2g-2) \right]^2$. Our results hold for any Schr\"odinger operator, not just the Laplacian.
\end{abstract}

\section{Introduction}

Let $G(V,E)$ be a finite connected graph. We denote by $x \sim y$ that $(xy) \in
E$. The degree of a vertex $v$ will be denoted by $\deg(v)$. The
Laplace operator associated to $G$ is a linear operator $\Delta : \R^V
\to \R^V$ given by $\Delta f (x) = \sum_{x \sim y} f(x) - f(y)$ for
any function $f \in \R^V$. We shall consider the more general class of Schr\"odinger operators.
Let $M = (m_{xy})_{x,y\in V}$ be any symmetric matrix satisfying $m_{xy} < 0$ if $x \sim y$ and $m_{xy} = 0$ otherwise. The diagonal entries $m_{xx}$ can be arbitrary. We denote again by $\Delta : \R^V \to \R^V$ the operator given by $\Delta f(x) = \sum_y m_{xy} f(y)$. 
Let us denoted the eigenvalues of $\Delta$ by $\la_1 <
\la_2 \leq \la_3 \leq \dots \leq \la_{|V|}$, and an eigenfunction
corresponding to $\la_i$ by $u^{(i)}$. (By the Perron-Frobenius theorem the multiplicity of $\la_1$ is 1, since $G$ is connected.) 

Let us fix an eigenfunction $u = u^{(n)}$. The vertices where $u$
vanishes are usually referred to as \textit{nodes}. A strong nodal
domain for $u$ is a maximal connected induced subgraph $D \leq G$ on
which $u$ is either strictly positive or strictly negative. Let
$D_1,D_2,\dots, D_t$ be the list of strong nodal domains. Davies et al
show in \cite{Davies} that $t \leq n + r - 1$ where $r$ is the
multiplicity of $\la_n$. We are interested in an upper bound for $t$
that involves the genus of the graph instead of the multiplicity of
$\la_n$. 

\begin{theorem}~\label{nodaldomainthm}
If the maximum degree is $d$ in $G$ then $t \leq d\cdot (n-1)$.
If the graph is 3-connected and $g$ denotes its genus then $t \leq 6(n -1) + 14(2g-2)$.
\end{theorem}

\begin{remark}
\begin{itemize}
\item It has been observed in \cite{Davies} that the star-graph on $N+1$ vertices behaves badly in terms of these type of questions. It has only three different eigenvalues: $\la_1 = 0, \la_2= \dots = \la_N =  1$ and $\la_{N+1} = N+1$. Furthermore any eigenfunction for $\la_2$ has exactly $N$ strong nodal domains. This shows that the first statement of Theorem~\ref{nodaldomainthm} is sharp.

\item The double-star $K_{2,N}$ has similar properties: 1 is an eigenvalue of multiplicity $N$, and any eigenfunction has $N$ strong nodal domains, while the genus is still 0. This shows that 3-connectedness is essential if we want an upper bound depending only on $n$ and $g$ in the second statement. 

\item One might then think that  the triple star $K_{3,N}$ could be a 3-connected counter-example. It is not, however, since its genus becomes suddenly large.

\end{itemize}
\end{remark}

Cheng~\cite{Cheng} proved that on a smooth surface of genus $g$ the
multiplicity of $\la_n$ is bounded by $(n+2g+1)(n+2g+2)/2$. The idea of his
proof is to use the high multiplicity to obtain an eigenfunction which
vanishes to a high order. This function will have a lot of sign
changes near this zero, and hence it will have many nodal domains. But
the number of nodal domains is limited by Courant's original nodal
domain theorem.  
Using our discrete version of the nodal domain theorem we can adapt
Cheng's approach for the graph case. However an extra assumption is
needed for our graph. 

\begin{defin}\label{volumegrowthdef}
A graph $G$ satisfies the quadratic volume-growth condition
$VG$ if for any $D \subset V$ such that $|D| \leq |V|/2$ we
have $|\partial D| \geq \sqrt{|D|}$. Here $\partial D$ denotes the outer
vertex-boundary of $D$, that is, those vertices of $V \setminus D$
that are adjacent to $D$.
\end{defin}

\begin{theorem}\label{multiplicityboundthm}
If $G$ is a 3-connected graph that satisfies $VG$ then the
multiplicity of $\la_n$ is at most $2\left[
  6(n-1) + 15(2g-2) \right]^2$ where $g$ is
the genus of $G$.
\end{theorem}

\begin{remark}
As the volume growth condition is used only at the very last step of the proof, it could be easly replaced by alternative versions, yielding sligtly different bounds in Theorem~\ref{multiplicityboundthm}.
\end{remark}

\section{Nodal geography}

Let us fix our graph $G$. Let $\la_n$
be the $n$-th eigenvalue of the Laplacian, and let $u = u^{(n)}$ be an
eigenfunction for $\la_n$. We may assume without loss of generality
that $\la_{n-1} < \la_n$, and fix pairwise orthogonal eigenfunctions
$u^{(1)},\dots, u^{(n-1)}$ corresponding to $\la_1,\dots,\la_{n-1}$. 

Let $\cD = \{D_1,D_2,\dots, D_t\}$ be the set of strong nodal domains of $u$. We start
by analyzing the relative location of these domains. We say that two
domains $D_1,D_2$ are adjacent if there is an edge $v_1 \sim v_2$ such
that $v_1 \in D_1$ and $v_2 \in D_2$. This of course implies that the
sign of $u$ on $D_1$ is different from that on $D_2$. This defines a
graph on the set of domains. 

Let us take any connected component of this graph, and take the union
of the corresponding domains. We shall call this a (nodal)
\textit{region} of $u$. Each region consist of one or more domains. It is
clear from the definition, that any vertex in the boundary of a region is a node.
We call a region \textit{small} if it consist of a single strong domain. Otherwise we call it \textit{large}.

We shall group the regions into larger compounds which we call (nodal)
\textit{islands} of $u$. Similarly to regions, we are going to
distinguish between \textit{small} islands - meaning they consist of a
single strong domain - and \textit{large} islands, which contain more
than one strong domain. The construction of islands is done
recursively. At the beginning each region is an island on its own
(either small or large, depending on the type of the region). In one
step we look for a node which is adjacent to exactly two different
islands, at least one of which has to be a small island, and unite
these two islands into one big island. (The result is then neccessarily a large island.) We repeat this step as long as
there are islands to unite. Let $I_1,I_2,\dots,I_s$ denote the final
list of islands. The set of small islands will be denoted by $\cS$ and
the set of large islands by $\cL$. The number of strong domains in an
island $I$ shall be denoted by $t(I)$.

\begin{claim}
Any node adjacent to a small island has to be adjacent to at least 3 different islands.
\end{claim}
\begin{proof}
Let us look at a small island $I$. If $v$ is a node adjacent to $I$ then the function $u$ is non-zero at a neighbor of $v$. But since $\Delta u (v) = \la_n u(v) = 0$, there must be another neighbor of $v$ where $u$ is of the opposite sign. This other vertex cannot be in $I$ since $I$ consist of a single strong domain. Hence it must belong to a different island. Then, by the definition of the islands the node $v$ must be adjacent to at least three islands.
\end{proof}

Let $V_0 \subset V$ denote the set of nodes adjacent to at least one small island.
Let us consider now the $t$-dimensional real Euclidean vector space
$\R^\cD$ with the standard scalar product $(f,g) = \sum_i
f(D_i)g(D_i)$, and the $s$ dimensional subspace $W \leq \R^\cD$
consisting of functions that are constant on the domains of each island. For any node $v \in V_0$ let $\fii_v \in W$ denote the function defined by 
\[ \fii_v(D) = \frac{1}{t(I(D))}\sum_{w \in
  I(D)}m_{vw} u(w).\]  Here $I(D)$ denotes the island in which the domain $D$ lies. The function $\fii_v$ is made so that it is automatically constant on each island.  Since $\Delta u(v) =0$, each $\fii_v$ is orthogonal to the constant 1 function.

\begin{lemma}  The dimension of the subspace $W_0 = \langle \fii_v : v \in V_0 \rangle \leq W$ is at least 
\begin{enumerate}
\item $|\cS| / d$ where $d$ denotes the maximum degree of $G$,
\item $\frac{1}{6}(|\cS|  - 14(2g-2))$  if the $G$ is 3-connected and $g$ denotes the genus of $G$.
\end{enumerate}
\end{lemma}

\begin{proof}
Both parts are proved by successively picking nodes $v_1, v_2, \dots \in V_0$ with the property that for every $i$ the node $v_i$ is adjacent to a small island that was not adjacent to any previously picked node. If $I \in \cS$ and $v$ is a node adjacent to $I$ then $\fii_{v_i}(I) \neq 0$. Thus our process guarantees that all the $\fii_{v_i}$ are independent. 

For the first part the greedy algorithm generates a good sequence $v_1, v_2, \dots$. In each step we find a small island that is not adjacent to any of the previously selected nodes, and choose any adjacent node as the next $v_i$. This way the number of small islands we can choose from decreases at most by $d$, hence the sequence of $v_i$ will be of length at least $|\cS|/d$.

For the second part we use a similar greedy algorithm. The idea is that for a fixed genus there is always a vertex of degree at most six, unless the graph is very small. Let us contract each small island to a point by contracting the edges of an arbitrary spanning tree of the island. Denote the resulting set of points by $W = \{w_1,\dots, w_{|\cS|}\}$. Let us only keep the subgraph spanned by $V_0 \cup W$ and delete all loop and multiple edges and in general any edge not running between $V_0$ and $W$. This way we get a new bipartite graph $H$ that is still embedded in $\Sigma_g$. Since $G$ was 3-connected, this means that every small island had to have at least 3 adjacent nodes in $V_0$. In $H$ this simply means that the degree of each $w_i$ is at least 3. 
The proof of the following statement will be given below.

\begin{claim}\label{smalldegree} If $|W| > 14(2g-2)$ then there is a vertex $v \in V_0$ whose degree is at most 6.
\end{claim}

This is all we need for our greedy algorithm to work: if $|W| \leq 14(2g-2)$ there is nothing to prove. On the other hand if $|W| > 14(2g-2)$ then by the claim there is a vertex $v  \in V_0$ with small degree. Let us choose $v_1 = v$ and remove $v_1$ and all its neighbors from $H$. This cannot increase the genus of the graph. We repeat the process until the size of $W$ shrinks below $14(2g-2)$. In each step we lose at most 6 vertices from $W$ hence we get at least $\frac{1}{6}(|W|  - 14(2g-2))$ = $\frac{1}{6}(|\cS|  - 14(2g-2))$ independent $\fii_v$ functions, as stated.
\end{proof}

\begin{proof}[Proof of Claim~\ref{smalldegree}]
Take the minimal genus representation of $H$. Then every face has to be a disc. Since the graph is bipartite and has no multiple edges, each face is an even cycle of length at least 4. If it is longer, we can cut it into smaller faces of length exactly 4 by drawing some of the diagonals, and keeping the graph bipartite. Finally we can transform the graph in the following way: on each face connect the two vertices belonging to $W$ by a dotted diagonal. The dotted edges form a graph embedded in $\Sigma_g$ whose vertex set is $W$ and the faces correspond exactly to the vertices of $V_0$. Denote the new graph by $H_1$.

Assume every degree in $V_0$ is at least 7, that is, each face of $H_1$ has at least 7 sides. 
Hence for this graph $e \geq 7f/2$, and $e \geq 3v/2$ since each vertex has degree at least 3. Multipying the first bound by 4, the second by 10 and adding them up we get
\[ 14f +  15v \leq 14e = 14f + 14v + 14(2g-2),\] that is $|W| = v \leq 14(2g-2)$ and this completes the proof.
\end{proof}

\begin{claim}~\label{ybecsles} Recall that $t$ denoted the total number of strong domains. Let $y$ denote the codimension of $W_0$ in $W$. Then we have
\begin{enumerate}
\item $y \leq \frac{d-1}{d} t$ where $d$ denotes the maximum degree of $G$,
\item $y \leq \frac{5}{6}t + \frac{14}{6}(2g-2)$  if the $G$ is 3-connected and $g$ denotes the genus of $G$.
\end{enumerate}
\end{claim}

\begin{proof}
Notice that each large island contains at least two strong domains, hence $t \geq 2|\cL| + |\cS|$. On the other hand by the lemma in case a) we have $y = |\cL| + |\cS| - \dim W_0 \leq 2\frac{d-1}{d}|\cL| + \frac{d-1}{d}|\cS| \leq \frac{d-1}{d} t$, and case b) is entirely analogous. 
\end{proof}

\begin{defin}
Let $\psi_1, \dots, \psi_y $ denote a basis of the orthogonal complement of $W_0$ in $W$. 
\end{defin}

\section{Proof of Theorem~\ref{nodaldomainthm}}

We use the notation from the previous section. Let $w_i : V \to \R$ be defined by 
\[ w_i(v) = \left \{ 
\begin{array}{rl}
u(v) & \mbox{ if $v \in D_i$} \\
0 & \mbox{ otherwise.}
\end{array} \right.
\]

Let us define $f = \sum c_i w_i$. Suppose we can choose the
coefficients such that $(f,f) = 1$ and $f$ is orthogonal to
$u^{(1)},\dots, u^{(n-1)}$, furthermore the function $c: \cD \to \R$
is orthogonal to $\psi_1,\dots, \psi_y$. We will follow closely the
approach of \cite{Davies} to show that these constraints imply that
all the $c_i$'s are equal to zero, which is a contradiction. The proof
goes in three steps. First we show that the $c_i$'s are constant in
each region, then in each island. Finally using orthogonality to the
$\psi_i$'s we get all the $c_i$'s are zero. The first step is
explicitly, the second is implicitly contained in \cite{Davies}, but
we repeat the arguments here to remain self-contained.

\begin{lemma} If $f = \sum c_i w_i$ is orthogonal to
  $u^{(1)},\dots,u^{(n-1)}$ then $\Delta f = \la_n f$, and for any two
  adjacent strong domains $D_i,D_j$ we have $c_i = c_j$.  
\end{lemma}

\begin{proof} We use Duval and Reiner's \cite{Duval} formula, which can be verified by straightforward computation. For any self-adjoint operator $A$:

\[ (f,Af) = \sum_{i=1}^t c_i^2(w_i, Au) - \frac{1}{2}\sum_{i,j = 1}^t (c_i - c_j)^2 (w_i, A w_j).\]

If we choose $A = \Delta - \la_n I$ then since $(f,f) = 1,Au = 0$ and for $i \neq j$ the product $(w_i,A w_j) = (w_i, \Delta w_j)$ we get

\[ (f,\Delta f) - \la = -\frac{1}{2} \sum_{i,j = 1}^t (c_i - c_j)^2 (w_i, \Delta w_j).\]

It is easy to see that $(w_i, \Delta w_j) = 0$ if $D_i$ and $D_j$ are
not adjacent. If they are, then $w_i$ and $w_j$ have different signs,
hence each non-zero term in $(w_i, \Delta w_j)$ is a product of a
positive and two negative numbers. 

So we have $(f, \Delta f) \leq \la_n$. On the other hand by the
well-known min-max principle $(f, \Delta f) \geq \la_n (f,f)$  if $f$
is orthogonal to the first $n-1$ eigenfunctions. Hence in our case $\la_n
\leq (f,\Delta f) \leq \la_n$. This implies by the same min-max
principle that $\Delta f = \la_n f$. On the other hand it also implies
that $(c_i - c_j)^2 (w_i, \Delta w_j) = 0$ for all $i,j$. If $D_i$ and
$D_j$ are adjacent, the argument above shows that in fact $(w_i,
\Delta w_j) > 0$, so we must have $c_i = c_j$. This completes the proof.
\end{proof}

\begin{lemma} The $c_i$'s are constant in each island.
\end{lemma}

\begin{proof} By the previous lemma we see that the $c_i$'s are
  constant in each region. We prove this lemma recursively as the
  islands were formed. At the beginning of the process each region is
  an island, hence the statement is true. The only thing we have to
  check is whenever two islands are merged into a larger island, the
  statement remains true. So lets consider a particular step of the
  process when two islands $I,J$ are merged into one large island. By
  induction we know that $c$ is constant on $I$ and on $J$. By the
  definition of the island forming process, at this time there must be
  a node $v$ which is adjacent to only these two islands. We know by
  the previous lemma, that $\Delta f = f$ and $\Delta u = u$. Let us
  write down what this precisely means for the node $v$.  Let
\[A = \sum_{x \in I} m_{vx} u(x); \qquad B = \sum_{x \in J} m_{vx}u(x).\] Since $u(v) = f(v) = 0$ we get regardless of the
value of $\la_n$ that
\[ c_I A + c_J B = (\Delta f)(x) = 0 = (\Delta u)(x)
= A + B.\]
 Since either $I$ or $J$ had to be a small island at this step of the process, either $A$ or $B$ has to be non-zero. But this implies the other being non-zero as well, and simple computation shows that this implies $c_I = c_J$.

We have showed that in each step when two islands are united, the
function $c$ remains constant in each island, hence this holds at the
end as well.
\end{proof}

We have shown that if we regard the coefficients $c_i$ as a function
$c: \cD \to \R$ then actually $c \in W$.

\begin{lemma} $c$ is orthogonal to $W_0$.
\end{lemma}

\begin{proof} Let $v\in V_0$ be a node (which is by definition adjacent to at least one small island). Let $J_1,\dots, J_p$ denote all the islands adjacent to $v$, and for each $j$ let 
\[A(j) = \sum_{x \in I_j} m_{vx}u(x) = \sum_{D \in J_j} \phi_v(D) \] where $D$ runs over all strong domains in the island $J_j$. The second equation holds by the definition of  $\fii_v$. Let us temporarily denote by $c(J)=c(D)$ the value of $c$ on any domain $D \in J$. We may do this, since $c$ is known to be constant on each island. Now similarly to the previous lemma we have 
\[ 0 = \sum_{x} m_{vx} f(x) = \sum_{j=1}^p c(J_j)A(j) = \sum_{j=1}^p \sum_{D \in J_j} c(D)\fii_v(D) = \sum_{i=1}^t c_i \fii_v(D_i).\]  As this holds for every $v \in V_0$, hence for every $\fii_v$ spanning $W_0$, we have the desired orthogonality.
\end{proof}

Since $c$ is also orthgonal to $\psi_1,\dots, \psi_y$ which is the
orthogonal complement of $W_0$ in $W$, this means that $c$ is
orthogonal to $W$. This together with $c \in W$ implies that $c = 0$,
contradicting our assumption.

Hence the $n+y-1$ orthogonality conditions imply all the $c_i$'s are zero, hence the number of strong domains $t$ is at most $n+y-1$. Using Claim~\ref{ybecsles} simple computation shows that in case a) we get $t \leq d(n-1)$ while $t \leq 6(n-1) + 14(2g-2)$ follows in case b). This completes the proof of the theorem.

\section{Proof of Theorem~\ref{multiplicityboundthm}}

Let $g$ denote the genus of $G$, and let us fix an embedding of $G$ into
$\Sigma_g$, the closed oriented surface of genus $g$. Let us fix the
$n$-th eigenvalue of the Laplacian $\la = \la_n$, and assume that it
has multiplicity $r$. This means there are $r$ linearly independent
eigenfunctions $f_1,\dots, f_r$ for $\la$. Combining these functions
we will try to create an eigenfunction which has many strong nodal
domains. 

First of all pick a set of $r$ vertices $R = \{v_1,\dots,v_r\}$ which
exhibit the independence of the functions $f_1,\dots, f_r$. Next
choose a connected subgraph $W' \subset V$ of size $|W'| = r/2$. The $W'$
and $R$ sets may overlap.

\begin{claim}
There is a linear combination $u = \sum a_i f_i$ that vanishes on
$W'$ but is non-zero on at least half of $R$.
\end{claim}

\begin{proof} Those eigenfunctions that vanish on $W'$ constitute an
  $r/2$ dimensional linear subspace of $\langle f_1,\dots, f_r
  \rangle$. Suppose that each of these functions vanishes on more than
  $r/2$ points of $R$. The set of eigenfunctions that vanish on a
  fixed vertex set of size $r/2+1$ is an $r/2-1$ dimensional subspace
  of $\langle f_1,\dots, f_r \rangle$. Hence we could cover an $r/2$
  dimensional space with finitely many $r/2-1$ dimensional ones, which
  is clearly impossible. Hence the desired linear combination exists.
\end{proof}

Let $W \subset V$ denote the connected component of nodes of $u$ that
contains $W'$. Let $Z = \partial (V \setminus W)$ the inner
vertex-boundary of $W$. 

\begin{claim}\label{zbound} $|Z| \geq \sqrt{r/2}-1$.
\end{claim}

\begin{proof}
Either $W$ or $V \setminus W$ contains at most half of all the
vertices. In the second case by the volume-growth property $|\partial
(V\setminus W)| \geq \sqrt{|V \setminus W|} \geq \sqrt{r/2}$.
In the first case we apply the growth estimate to $W \setminus
Z$. Obviously $\partial (W \setminus Z) = Z$, hence $|Z|^2 \geq |W| -
|Z| \geq r/2 - |Z|$. From this we get $(|Z|+1)^2 > |Z|^2 + |Z| \geq
r/2$ and the claim follows.
\end{proof}

Each vertex in $Z$ is adjacent to a non-node of $u$, hence it has to
be adjacent to at least a positive and a negative vertex.

Let us consider $G^*$, the dual graph of $G$ on $\Sigma_g$. On each
face of $G^*$, let us record the sign of $u$, whether it is plus,
minus or zero. 

Let us remove each edge from $G^*$ that has the same sign recorded on
its two sides. Any time we find a vertex of degree two, let us replace
the two edges with a single edge, thereby removing the vertex. If we
find isolated or degree one vertices, let us remove those too. It is
clear, that after this process each face of the remaining graph
corresponds to a strong domain of $u$ or to a connected group of nodes
of $u$. In particular there is the face corresponding to the nodes in
$W$. By the construction this face now has at least $|Z|$ sides and
$|Z|$ vertices. This is because if we trace the boundary of this
region from the outside, we encounter at least $|Z|$ sign-changes, one
at each vertex of $Z$. 

Next we remove all the faces that correspond to nodes if $u$. If such
a face is a $p$-gon, then we contract it to a single vertex which will
have degree at least $p$. If the face had more than one boundary component,
then we remove the face from $\Sigma$, glue a disc to each boundary
component, and then contract each of these new faces to single
vertices as above. This step might disconnect the surface or decrease
its genus, but that will only be to our advantage. If in this process any vertices of degree 2 were created, we remove
them as above. 

Let us see what remains: each face now corresponds to precisely one
strong domain of $u$. Since adjacent domains have opposite sign, this
means that every vertex of the remaining graph has an even degree,
which cannot be 2, hence each degree is at least 4. There is one
special vertex that has degree at least $|Z|$. (This came from
contracting our distinguished face.) The graph is drawn on a disjoint
union of surfaces whose total genus is at most $g$. By connecting the
surface-components with small tubes we can get a single surface
$\Sigma'$ of genus at most $g$ in which the graph is embedded. Euler's
formula now says that $e \leq 2g-2 + f +v$ where $e$ is the number of
edges, $f$ the number of faces and $v$ the number of vertices. On the
other hand $e \geq (|Z| + 4(v-1))/2$ by simple counting.
Putting this together we get
\[ f \geq 2-2g + v + |Z| - 2 \geq |Z| + 1 - 2g.\]
On the other hand from Theorem~\ref{nodaldomainthm} we know that $f
\leq 6(n-1) + 14(2g-2)$. Hence by Claim~\ref{zbound} we get $\sqrt{r/2} -
1 \leq |Z| \leq 6(n-1) + 15(2g - 2) -1$, from which $r \leq 2\left[
  6(n-1) + 15(2g-2) \right]^2$, exactly what we had to prove.

\end{document}